\documentclass[english,letter paper,12pt,reqno]{article}
\usepackage{etex} 
\usepackage{array}
\usepackage{bussproofs}
\usepackage{stmaryrd}
\usepackage{amsmath, amscd, amssymb, mathrsfs, accents, amsfonts,amsthm}
\usepackage[all]{xy}
\usepackage{mathtools} 
\usepackage{tikz}
\usetikzlibrary{calc}

\DeclarePairedDelimiter\ket{\lvert}{\rangle}
\DeclarePairedDelimiterX\braket[2]{\langle}{\rangle}{#1 \delimsize\vert #2}
\DeclarePairedDelimiterX\inner[2]{\langle}{\rangle}{#1,#2}

\definecolor{Myblue}{rgb}{0,0,0.6}
\usepackage[a4paper,colorlinks,citecolor=Myblue,linkcolor=Myblue,urlcolor=Myblue,pdfpagemode=None]{hyperref}

\SelectTips{cm}{}

\newtheorem{theorem}{Theorem}[section]
\newtheorem{proposition}[theorem]{Proposition}
\newtheorem{lemma}[theorem]{Lemma}
\newtheorem{corollary}[theorem]{Corollary}

\newtheoremstyle{example}{\topsep}{\topsep}
	{}
	{}
	{\bfseries}
	{.}
	{2pt}
	{\thmname{#1}\thmnumber{ #2}\thmnote{ #3}}
	
	\theoremstyle{example}
	\newtheorem{definition}[theorem]{Definition}
	\newtheorem{example}[theorem]{Example}
	\newtheorem{remark}[theorem]{Remark}

\numberwithin{equation}{section}

\def\sh{\operatorname{Sh}}
\def\res{\operatorname{Res}}

\def\Ker{\operatorname{Ker}}

\def\Hom{\operatorname{Hom}}

\def\vacu{\ket{\emptyset}}

\DeclareMathOperator{\End}{End}

\DeclareMathOperator{\Spec}{Spec}

\DeclareMathOperator{\Sym}{Sym}
\DeclareMathOperator{\LC}{LC}
\def\int{\bold{int}}
\def\contract{\;\lrcorner\;}

\begin{document}

\def\ScoreOverhang{1pt}

\def\Res{\res\!}
\newcommand{\ud}[1]{\operatorname{d}\!{#1}}
\newcommand{\Ress}[1]{\res_{#1}\!}
\newcommand{\cat}[1]{\mathcal{#1}}
\newcommand{\lto}{\longrightarrow}
\newcommand{\xlto}[1]{\stackrel{#1}\lto}
\newcommand{\mf}[1]{\mathfrak{#1}}
\newcommand{\md}[1]{\mathscr{#1}}
\newcommand{\church}[1]{\underline{#1}}
\newcommand{\prf}[1]{\underline{#1}}
\newcommand{\den}[1]{\llbracket #1 \rrbracket}
\def\l{\,|\,}
\def\sgn{\textup{sgn}}
\def\cont{\operatorname{cont}}

\title{On Sweedler's cofree cocommutative coalgebra}
\author{Daniel Murfet}

\maketitle

\begin{abstract} We give a direct proof of a result of Sweedler describing the cofree cocommutative coalgebra over a vector space, and use our approach to give an explicit construction of liftings of maps into this universal coalgebra. The basic ingredients in our approach are local cohomology and residues.
\end{abstract}

\section{Introduction}

Let $k$ be an algebraically closed field of characteristic zero. Given a vector space $V$ there is a universal cocommutative coalgebra ${!}V$ mapping to $V$ which is sometimes called the cofree cocommutative coalgebra generated by $V$. This is a classical construction going back to Sweedler, and an explicit description follows from his results in \cite{sweedler}: for $V$ finite-dimensional the universal cocommutative coalgebra may be presented as a coproduct
\begin{equation}\label{eq:presentation_intro}
{!} V = \bigoplus_{P \in V} \Sym_P(V)
\end{equation}
where $\Sym_P(V) = \Sym(V)$ is the symmetric coalgebra. The universal map $d: {!} V \lto V$ is defined on components $v_i \in V^{\otimes i}$ by the formula
\[
d|_{\Sym_P(V)}: \Sym_P(V) \lto V, \qquad (v_0,v_1,v_2,\ldots) \longmapsto v_0 P + v_1\,.
\]
While not written explicitly in \cite{sweedler} this is a straightforward exercise using the structure theory of coalgebras developed there. We do this exercise in Appendix \ref{section:compare_sweedler}. However, since Sweedler's description \eqref{eq:presentation_intro} of ${!} V$ seems not to be well-known despite the extensive discussion of cofree cocommutative coalgebras in models of logic \cite{blute,hyland,mellies2} and elsewhere it seems worthwhile to give a more direct proof. 

Our main new result is an explicit description of the lifting of a linear map ${!} W \lto V$ to a morphism of coalgebras ${!} W \lto {!} V$ (Theorem \ref{theorem:describe_lifting}). In \cite{murfet_ll} we explain in more detail why we are interested in such liftings, and give several examples of how the theorem may be used to write down explicit denotations of proofs in linear logic.

Our approach is based on an isomorphism of coalgebras ($n = \dim(V)$)
\begin{equation}\label{eq:lc}
\Sym_P(V) \cong H^{n}_{\mf{m}_P}(\Sym(V^*), \Omega^n_{\Sym(V^*)/k})
\end{equation}
of the symmetric coalgebra with the local cohomology module of top-degree differential forms at the maximal ideal $\mf{m}_P \subseteq \Sym(V^*)$ corresponding to $P$. Elements of this local cohomology module are equivalence classes of meromorphic differential forms at $P$, and both the universal map $d: {!} V \lto V$ and the counit ${!} V \lto k$ are defined in terms of residues.\footnote{One motivation for phrasing things in terms of residues and local cohomology is the hope that the connection between cofree coalgebras, differential operators and residues extends beyond the very simple case of polynomial rings which arises in the case of $k$ a field. This would be very interesting.} This approach has the advantage of being manifestly coordinate-free (because residues are) in contrast to an alternative construction using the isomorphism ${!}( V_1 \oplus V_2 ) \cong {!} V_1 \otimes {!} V_2$ and a reduction to the case of $V$ one-dimensional (see Remark \ref{remark_additive_iso}).

\vspace{0.2cm}

\emph{Acknowledgements.} Thanks to Nils Carqueville and Jesse Burke, and to the referee who made several helpful suggestions on how to improve the exposition.

\section{The universal cocommutative coalgebra}\label{section:expmod}

Let $k$ be an algebraically closed field of characteristic zero.\footnote{Note that the original approach of Sweedler in Appendix \ref{section:compare_sweedler} works in any characteristic, but the statement is more complicated.} Throughout all coalgebras are coassociative, cocommutative and counital. We begin with $V$ finite-dimensional, but the results will immediately generalise; see Section \ref{section:infinite_dim}. For general background on commutative algebra we recommend \cite{eisenbud}, and for the basic theory of coalgebras \cite{sweedler}. Other references on cofree coalgebras are \cite{getzler, anel, block-leroux, hazewinkel, smith, barr}.
\\

Let $V$ be a finite-dimensional vector space, $R = \Sym(V^*)$ the symmetric algebra of the dual $V^*$. Since $k$ is algebraically closed the maximal ideals of $R$ are in canonical bijection with the elements of $V$, and for $P \in V$ we denote the corresponding maximal ideal $\mf{m}_P$. 

We define $\LC(V,P)$ to be the local cohomology module
\[
\LC(V,P) := H^n_{\mf{m}_P}(\Sym(V^*), \Omega^n_{\Sym(V^*)/k})
\]
where $n = \dim(V)$. This is an $R$-module which is infinite-dimensional as a $k$-vector space, elements of which are equivalence classes of meromorphic differential $n$-forms at $P$.

\begin{example} Suppose $V = k \cdot e$ is one-dimensional, with dual basis $x = e^*$ so $R = k[x]$. Then $P \in k$ and $\mf{m}_P = (x-P)$. Let $M$ be the space of meromorphic $1$-forms which are regular away from $P$,
\[
M = \Big\{ \frac{Q(x)}{(x-P)^r} \ud{x} \l Q(x) \in k[x], r \ge 0 \Big\}\,.
\]
Let $M' = k[x] \ud{x}$ be the subspace of regular forms, then $\LC(V,P)$ is the quotient space
\[
\LC( V, P ) = M / M'\,.
\]
The class of a meromorphic form $\frac{Q(x)}{(x-P)^r} \ud{x}$ is denoted by
\begin{equation}\label{eq:gen_frac_1d}
\left[ \frac{Q(x) \ud{x}}{(x-P)^r} \right] \in \LC(V,P)\,.
\end{equation}
From this description it is clear that $(x-P) \cdot \left[ \frac{Q(x) \ud{x}}{(x-P)} \right] = 0$ and that a $k$-basis for $\LC(V,P)$ is given by the equivalence classes $\left[ \frac{\ud{x}}{(x-P)^r} \right]$ for $r > 0$.
\end{example}

In general, elements of $\LC(V,P)$ are equivalence classes of meromorphic $n$-forms regular away from $P$, modulo those forms which are regular at $P$. These equivalence classes are referred to in the literature as \emph{generalised fractions} \cite{Lipman84,Kunz08}, and such a class
\begin{equation}\label{eq:gen_frac_woh}
\left[ \frac{f \ud{r_1} \wedge \cdots \wedge \ud{r_n}}{t_1, \ldots, t_n} \right] \in \LC(V,P)
\end{equation}
may be associated to the following data:
\begin{itemize}
\item An element $f \in R$ and a sequence $r_1,\ldots,r_n \in R$,
\item A sequence of elements $t_1,\ldots,t_n \in R$ which have an isolated common zero at $P$, or equivalently, they form a regular sequence in $R_{\mf{m}_P}$.
\end{itemize}
These classes in local cohomology satisfy natural identities generalising the one-dimensional case, for instance
\begin{gather*}
t_i \cdot \left[ \frac{f \ud{r_1} \wedge \cdots \wedge \ud{r_n}}{t_1^{a_1}, \ldots, t_i^{a_i}, \ldots, t_n^{a_n}} \right] = \left[ \frac{f \ud{r_1} \wedge \cdots \wedge \ud{r_n}}{t_1^{a_1}, \ldots, t_i^{a_i-1}, \ldots, t_n^{a_n}} \right]\,,\\
t_i \cdot \left[ \frac{f \ud{r_1} \wedge \cdots \wedge \ud{r_n}}{t_1, \ldots, t_n} \right] = 0\,.
\end{gather*}
It is important to note that the definition of local cohomology and the generalised fraction \eqref{eq:gen_frac_woh} associated to the data $f,r_1,\ldots,r_n,t_1,\ldots,t_n$ does not require a choice of coordinates. However, by choosing coordinates we can reduce the general case to the one-dimensional case, and in this way obtain a $k$-basis for $\LC(V,P)$:

\begin{example}\label{example:coordinates_V} Let $e_1,\ldots,e_n$ be a basis of $V$ with dual basis $x_i = e_i^*$ and identify $R$ with the polynomial ring $k[x_1,\ldots,x_n]$. Then for $P \in V$ with coordinates $(P_1,\ldots,P_n)$ in this basis there is a corresponding maximal ideal
\[
\mf{m}_P = (x_1 - P_1, \ldots, x_n - P_n) \in \Spec(R)
\]
and if we set $z_i = x_i - P_i$ then there are generalised fractions
\begin{equation}\label{eq:basis_of_gen_frac}
\left[ \frac{\ud{z_1} \wedge \cdots \wedge \ud{z_n}}{z_1^{a_1+1}, \ldots, z_n^{a_n+1}} \right] \in \LC(V, P)\,.
\end{equation}
Moreover there is an isomorphism of $R$-modules
\begin{gather}
\LC( k \cdot e_1, P_1 ) \otimes \cdots \otimes \LC( k \cdot e_n, P_n ) \lto \LC( V, P )\,,\nonumber\\
\left[ \frac{\ud{z_1}}{z_1^{a_1+1}} \right] \otimes \cdots \otimes \left[ \frac{\ud{z_n}}{z_n^{a_n+1}} \right] \mapsto \left[ \frac{\ud{z_1} \wedge \cdots \wedge \ud{z_n}}{z_1^{a_1+1}, \ldots, z_n^{a_n+1}} \right]\label{eq:one_to_general}
\end{gather}
and the elements \eqref{eq:basis_of_gen_frac} for $a_1,\ldots,a_n \ge 0$ form a $k$-basis for $\LC(V,P)$. It will simplify the notation to write $\ud{\underline{z}} = \ud{z_1} \wedge \cdots \wedge \ud{z_n}$ and $\underline{z} = z_1, \ldots, z_n$ so that \eqref{eq:basis_of_gen_frac} may be abbreviated
\begin{equation}\label{eq:basic_residue_symbols}
\left[ \frac{1}{z_1^{a_1}, \ldots, z_n^{a_n}} \frac{\ud{\underline{z}}}{\underline{z}} \right] := \left[ \frac{\ud{z_1} \wedge \cdots \wedge \ud{z_n}}{z_1^{a_1+1}, \ldots, z_n^{a_n+1}} \right]\,.
\end{equation}
In terms of this basis, the $R$-module structure on $\LC(V,P)$ is defined by
\begin{equation}\label{eq:Ractionres}
z_j \cdot \left[ \frac{1}{z_1^{a_1}, \ldots, z_n^{a_n}} \frac{\ud{\underline{z}}}{\underline{z}} \right] = \left[ \frac{1}{z_1^{a_1}, \ldots, z_j^{a_j-1}, \ldots, z_n^{a_n}} \frac{\ud{\underline{z}}}{\underline{z}} \right]
\end{equation}
where we observe that the right hand side is zero if $a_j = 0$.
\end{example}

A fundamental result about local cohomology states that there is a scalar associated to each meromorphic form, its \emph{residue}, which is defined in a coordinate-independent way. When $k = \mathbb{C}$ and $n = 1$, this is the usual Cauchy integral around a pole. There is a vast generalisation to general schemes, known as the Grothendieck residue symbol \cite{residuesduality, lipman_notes} but we need only the simplest aspects of the theory.

We say that a sequence $z_1,\ldots,z_n \in R$ gives \emph{local coordinates at $P$} if the $\mf{m}_P$-adic completion of $R_{\mf{m}_P}$ is isomorphic to $k \llbracket z_1,\ldots,z_n \rrbracket$.


  
\begin{theorem}\label{theorem:compute_res} There is a canonically defined linear map
\[
\Res_{P}: \LC(V,P) \lto k
\]
with the property that for any sequence $z_1,\ldots,z_n \in R$ giving local coordinates at $P$,
\[
\Res_{P} \left[ \frac{1}{z_1^{a_1}, \ldots, z_n^{a_n}} \frac{\ud{\underline{z}}}{\underline{z}} \right] = \begin{cases} 1 & a_1 = \cdots = a_n = 0 \\ 0 & \textup{otherwise} \end{cases}
\]
\end{theorem} 
\begin{proof}
See \cite[\S 5.3]{Lipman01}, \cite[pp.64--67]{Lipman84} or \cite{Kunz08}. Notice that in the situation of Example \ref{example:coordinates_V} the isomorphism \eqref{eq:one_to_general} identifies $\Res_{P} = \Res_{P_1} \otimes \cdots \otimes \Res_{P_n}$ where
\[
\Res_{P_i} \left[ \frac{\ud{x}}{(x_i-P_i)^r} \right] = \delta_{r,1}\,.
\]
This defines a $k$-linear map $\Res_P$, and in order to prove that this is \emph{canonical} one has to check that it does not depend on the choice of local coordinates at $P$ (which is straightforward in this case, because $R$ is a polynomial ring).
\end{proof}

\begin{remark} In the case $k = \mathbb{C}$ the residue may be defined as the integral of a meromorphic $n$-form over an $n$-cycle which avoids its poles, see for instance \cite[Chapter V]{Griffiths}.
\end{remark}

\begin{definition} Given a $k$-algebra $S$, a linear map $\mu: S \lto k$ is called \emph{continuous} if it vanishes on an ideal of finite codimension, that is, if $\mu( I ) = 0$ for an ideal $I \subseteq S$ with $S/I$ a finite-dimensional $k$-vector space. We denote the module of all continuous maps by $\Hom_k^{\cont}(S,k)$, this module is often also written $S^{\circ}$ \cite[Chapter 6]{sweedler}.
\end{definition}

For a maximal ideal $\mf{m} \subseteq R$ an ideal in $R_{\mf{m}}$ has finite codimension if and only if it contains a power of $\mf{m}$. Moreover, note that since $R/\mf{m}^i \cong R_{\mf{m}}/\mf{m}^i R_{\mf{m}}$
\begin{equation}\label{eq:contmap}
\Hom_k^{\cont}(R_{\mf{m}},k) = \varinjlim_i \Hom_k( R/\mf{m}^i, k)\,.
\end{equation}

\begin{theorem}\label{theorem:lcwithcontmap} For $P \in V$ there is a canonical isomorphism of $R$-modules
\begin{gather}
\LC(V,P) \lto \Hom_k^{\cont}(R_{\mf{m}_P},k)\label{eq:local_duality}\\
\eta \longmapsto \Res_P( ( - ) \cdot \eta )\,.\nonumber
\end{gather}
\end{theorem}
\begin{proof}
This is a result of the theory of local duality, see \cite{Lipman01, Kunz08}. We appreciate that these references may be daunting to the non-specialist, so here is a sketch: one can check that with $\mf{m} = \mf{m}_P$ the map
\begin{equation}\label{eq:local_duality_1}
\Hom_R( R/\mf{m}^i, \LC(V,P) ) \lto \Hom_k( R/\mf{m}^i, k ), \qquad \alpha \longmapsto \Res_P \circ \alpha
\end{equation}
is an isomorphism, either abstractly \cite[Proposition 0.4]{Salas02} or by noting that the left hand side is the submodule of all elements of $\LC(V,P)$ killed by $\mf{m}^i = (z_1,\ldots,z_n)^i$ and a $k$-basis for this may be found among the generalised fractions in \eqref{eq:basis_of_gen_frac}. In this way we see easily that both sides of \eqref{eq:local_duality_1} have the same dimension, and injectivity of the map may be checked using Lemma \ref{lemma:residue_differentiates_0} below.

Any given generalised fraction is annihilated by all sufficiently high degree monomials in the $z_i$, as is apparent from \eqref{eq:Ractionres}, so the direct limit over all $i$ yields an isomorphism
\[
\LC(V,P) \cong \Hom_k^{\cont}(R_{\mf{m}_P}, k)\,,
\]
as claimed.
\end{proof}

Let $\cat{D}(R)$ denote the ring of $k$-linear differential operators on $R$. This may be defined as a subalgebra of $\End_k(R)$ using only the $k$-algebra structure on $R$ \cite{milicic}, but it may also be presented in a more straightforward way using the coordinates of Example \ref{example:coordinates_V} as the associative $k$-algebra generated by $x_i$ and $\partial_i = \frac{\partial}{\partial x_i}$ for $1 \le i \le n$ subject to the relations
\[
[x_i, x_j] = [\partial_i, \partial_j] = 0, \qquad [\partial_i, x_j] = \delta_{ij}\,.
\]
Observe that $D = \cat{D}(R)$ is a noncommutative ring, and that $R$ is a left $D$-module, with the $x_i$ acting by left multiplication and the $\partial_i$ acting as partial derivatives. Moreover $D$ is a filtered ring, with $D^{\le r}$ the differential operators of order $\le r$.

If $\mf{m} \subseteq R$ is a maximal ideal then $D^{\le r} \mf{m}^{i+r} \subseteq \mf{m}^i$, so that any differential operator $\rho \in D^{\le r}$ determines a linear map $R/\mf{m}^{i+r} \lto R/\mf{m}^i$ and thus an action on the direct limit \eqref{eq:contmap}. The action is induced by dualising the following map of inverse systems:
\[
\xymatrix{
\cdots \ar[r] & R/\mf{m}^{i+r} \ar[r] \ar[d]^\rho & \cdots \ar[r] & R/\mf{m}^{2+r} \ar[r]\ar[d]^\rho & R/\mf{m}^{1+r} \ar[d]^\rho\\
\cdots \ar[r] & R/\mf{m}^{i} \ar[r] & \cdots \ar[r] & R/\mf{m}^2 \ar[r] & R/\mf{m}
}
\]
This defines a \emph{right} $D$-module structure on $\Hom_k^{\cont}(R_{\mf{m}}, k)$ which sends a continuous map $\mu: R_{\mf{m}} \lto k$ factoring as $R_{\mf{m}} \lto R_{\mf{m}}/\mf{m}^i R_{\mf{m}} \cong R/\mf{m}^i$ followed by $\mu': R/\mf{m}^i \lto k$ to the following continuous map
\[
\xymatrix{
R_{\mf{m}} \ar[r] & R_{\mf{m}}/\mf{m}^{i+r}R_{\mf{m}} \cong R/\mf{m}^{i+r} \ar[r]^-{\rho} & R/\mf{m}^i \ar[r]^-{\mu'} & k
}\,.
\]
Using the isomorphism of Theorem \ref{theorem:lcwithcontmap} we conclude that:

\begin{lemma}\label{lemma:dmodulestructure} $\LC(V,P)$ is canonically a right $\cat{D}(R)$-module.
\end{lemma}

To describe the action more concretely, we use:

\begin{lemma}\label{lemma:residue_differentiates_0} For $f \in R$,
\begin{equation}\label{eq:residue_differentiates2}
\Res_P\left[ \frac{f}{z_1^{a_1}, \ldots, z_n^{a_n}} \frac{\ud{\underline{z}}}{\underline{z}} \right] = \frac{1}{{a_1}! \cdots {a_n}!} \frac{\partial^{a_1}}{\partial x_1^{a_1}} \cdots \frac{\partial^{a_n}}{\partial x_n^{a_n}}(f) |_{x = P}\,.
\end{equation}
\end{lemma}
\begin{proof}
We may assume $f$ is a monomial in the $z_i$. The action of $f$ on a generalised fraction is described by \eqref{eq:Ractionres} and consulting Theorem \ref{theorem:compute_res} we see that both sides of \eqref{eq:residue_differentiates2} vanish unless $f$ has $z_i$-degree $= a_i$ for every $i$. Thus we need only consider the case where $f$ is $\lambda z_1^{a_1} \cdots z_n^{a_n}$ for some $\lambda \in k$, but then both sides obviously agree.
\end{proof}

The action in Lemma \ref{lemma:dmodulestructure} may be described in the coordinates of Example \ref{example:coordinates_V} as follows: the generators $x_i = z_i + P_i$ act via the usual $R$-module structure, so
\begin{equation}
\left[ \frac{1}{z_1^{a_1}, \ldots, z_n^{a_n}} \frac{\ud{\underline{z}}}{\underline{z}} \right] \cdot x_i = \left[ \frac{1}{z_1^{a_1}, \ldots, z_i^{a_i-1}, \ldots, z_n^{a_n}} \frac{\ud{\underline{z}}}{\underline{z}} \right] + P_i \left[ \frac{1}{z_1^{a_1}, \ldots, z_n^{a_n}} \frac{\ud{\underline{z}}}{\underline{z}} \right]
\end{equation}
while the differential operators act by
\begin{equation}\label{eq:action_diff_op}
\left[ \frac{1}{z_1^{a_1}, \ldots, z_n^{a_n}} \frac{\ud{\underline{z}}}{\underline{z}} \right] \cdot \partial_i = (a_i + 1) \left[ \frac{1}{z_1^{a_1}, \ldots, z_i^{a_i+1}, \ldots, z_n^{a_n}} \frac{\ud{\underline{z}}}{\underline{z}} \right]\,,
\end{equation}
as can be checked by precomposing \eqref{eq:residue_differentiates2}, viewed as a functional of $f$, with the differential operator $\partial_i$ on $R$. 

At this point the notation is starting to get a bit awkward, so we introduce a second basis for $\LC(V,P)$ which is better suited to our purposes.

\begin{definition} The \emph{vacuum} at $P \in V$ is the generalised fraction
\[
\vacu_P := \left[ \frac{\ud{\underline{z}}}{\underline{z}} \right] = \left[ \frac{\ud{(x_1-P_1)} \wedge \cdots \wedge \ud{(x_n - P_n)}}{(x_1-P_1), \ldots, (x_n-P_n)} \right] \in \LC(V,P)\,.
\]
This element does not depend on the choice of basis $e_1,\ldots,e_n$ for $V$.
\end{definition}

\begin{definition}\label{def:symvtoder} There is a canonical map
\begin{gather*}
V \lto \cat{D}(R),\\
\sum_i a_i e_i \longmapsto \sum_i a_i \partial_i
\end{gather*}
and we denote the image of $\nu \in V$ by $\partial_{\nu}$. This map lands in the commutative subalgebra of $\cat{D}(R)$ generated as a $k$-algebra by the $\partial_i$, there is an extension to a morphism of algebras
\begin{equation}\label{eq:symvtoder}
\Sym(V) \lto \cat{D}(R)\,.
\end{equation}
\end{definition}

The right $\cat{D}(R)$-action together with \eqref{eq:symvtoder} induces a right $\Sym(V)$-module structure on $\LC(V,P)$. Since this ring is commutative we may write its action on the left.

\begin{definition} For any sequence $\nu_1,\ldots,\nu_s \in V$ there is an associated element
\begin{equation}\label{eq:s_particle_state}
\ket{ \nu_1, \ldots, \nu_s }_P := \partial_{\nu_1} \cdots \partial_{\nu_s} \vacu_P \in \LC(V,P)\,.
\end{equation}
Note that since the $\partial_{\nu_i}$ commute, the order of the $\nu_i$ is irrelevant. Using \eqref{eq:action_diff_op} we have, in terms of coordinates,
\begin{equation}\label{eq:defn_kets}
\ket{ e_{i_1}, \ldots, e_{i_s} }_P = a_1! \cdots a_n! \left[ \frac{1}{z_1^{a_1}, \ldots, z_n^{a_n}} \frac{\ud{\underline{z}}}{\underline{z}} \right]
\end{equation}
where $e_i$ occurs $a_i$ times among the $e_{i_t}$. The notation is derived from the fact that the $\partial_i$ and $x_i$ obey the canonical commutation relations, and may therefore be viewed as acting as creation and annihilation operators for particle states on the vacuum $\vacu_P$.
\end{definition}

\begin{lemma}\label{lemma:iso_fock} There is a linear isomorphism 
\begin{gather*}
\Sym(V) \lto \LC(V, P)\\
f \longmapsto f \cdot \vacu_P\,.
\end{gather*}
\end{lemma}
\begin{proof}
A bijection between the monomial basis of $\Sym(V)$ and the generalised fractions forming a basis for $\LC(V,P)$ is given by \eqref{eq:defn_kets}.
\end{proof}

The connection between differential operators and generalised fractions may now be clearly formulated in terms of \emph{integrating against} the latter. The following is a consequence of Lemma \ref{lemma:residue_differentiates_0}:

\begin{lemma}\label{lemma:residue_differentiates} For $f \in R$ and $\nu_1,\ldots,\nu_s \in V$
\begin{equation}\label{eq:residue_differentiates}
\Res_P\Big( f \ket{\nu_1,\ldots,\nu_s}_P \Big) = \partial_{\nu_1} \cdots \partial_{\nu_s}( f )|_{x = P}\,.
\end{equation}
In particular $\Res_P( f \vacu_P ) = f(P)$.
\end{lemma}

Given a subset of indices $I =  \{ i_1, \ldots, i_k \} \subseteq \{ 1, \ldots, s \}$ we write
\begin{equation}\label{eq:nu_I}
\ket{ \nu_I }_P := \ket{ \nu_{i_1}, \ldots, \nu_{i_k} }_P\,.
\end{equation}
By convention this is $\vacu_P$ if $I$ is empty. The complement of $I$ is denoted $I^c$.

\begin{lemma}\label{lemma:lc_coalgebra} $\LC(V,P)$ is a $k$-coalgebra with coproduct
\begin{equation}\label{eq:lc_coalgebra_1}
\Delta \ket{ \nu_1, \ldots, \nu_s }_P = \sum_{I \subseteq \{ 1, \ldots, s \}} \ket{ \nu_I }_P \otimes \ket{ \nu_{I^c} }_P
\end{equation}
and counit $\Res_P: \LC(V,P) \lto k$.
\end{lemma}
\begin{proof}
This may be deduced from \eqref{eq:contmap} since each $R/\mf{m}^i$ is a finite-dimensional commutative algebra, hence $\Hom_k( R/\mf{m}^i, k )$ is canonically a coalgebra. The induced coalgebra structure on the direct limit passes by Theorem \ref{theorem:lcwithcontmap} to a coalgebra structure on $\LC(V,P)$. On the basis elements \eqref{eq:defn_kets} this structure is easily checked to be given by \eqref{eq:lc_coalgebra_1}.
\end{proof}

For example, $\Delta \ket{\nu}_P = \vacu_P \otimes \ket{\nu}_P + \ket{\nu}_P \otimes \vacu_P$.

\begin{definition}\label{definition:bang} As a coalgebra, $!V$ is given by the coproduct
\begin{equation}
!V = \bigoplus_{P \in V} \LC(V,P)
\end{equation}
with its natural coproduct structure. The sum of the residue maps defines a linear map
\begin{equation}\label{eq:residue_der}
\Res\, = \sum_P \Res_P: {!}V \lto k
\end{equation}
which is the counit of this coalgebra.
\end{definition}

Next we define the universal map ${!} V \lto V$, which we refer to as the \emph{dereliction} map following the terminology used in the semantics of linear logic \cite{mellies}. For each $P \in V$ there is a canonical pairing
\[
\LC(V,P) \otimes_k R \lto k, \qquad \eta \otimes r \longmapsto \Res_P(r \eta)\,.
\]
Using $V^* \subseteq R$ the induced pairing $\LC(V,P) \otimes_k V^* \lto k$ corresponds under adjunction to a linear map $d_P: \LC(V,P) \lto V$, given by $d_P( \eta ) = \sum_i \Res_P( x_i \eta ) e_i$.

\begin{definition} The \emph{dereliction} map $d: {!}V \lto V$ is the linear map $d = \sum_P d_P$.
\end{definition}

In terms of coordinates, using the map of \eqref{eq:residue_der}, $d( \eta ) = \sum_{i=1}^n \Res\,( x_i \eta ) e_i$.

\begin{lemma}\label{lemma:dereliction_describe} $d \vacu_P = P, d \ket{\nu}_P = \nu$ for any $\nu \in V$, while $d \ket{\nu_1,\ldots,\nu_s}_P = 0$ for $s > 1$.
\end{lemma}
\begin{proof}
Both claims follow from Lemma \ref{lemma:residue_differentiates}:
\begin{gather*}
d \vacu_P = \sum_i \Res_P( x_i \vacu_P ) e_i = \sum_i P_i e_i = P\,.\\
d \ket{\nu}_P = \sum_i \Res_P( x_i \ket{\nu}_P ) e_i = \sum_i \partial_\nu( x_i ) e_i = \nu\,.
\end{gather*}
\end{proof}

Recall that the tensor algebra $TV = \bigoplus_{i \ge 0} V^{\otimes i}$ is a bialgebra with coproduct $\Delta$ and counit $\varepsilon$ determined by $\Delta(v) = 1 \otimes v + v \otimes 1$ and $\varepsilon(v) = 0$ for $v \in V$. More explicitly,
\begin{equation}
\Delta(v_1 \otimes \cdots \otimes v_s) = \sum_{i=0}^s \sum_{\sigma \in \operatorname{Sh}_{i,s-i}} ( v_{\sigma(1)} \otimes \cdots \otimes v_{\sigma(i)} ) \otimes ( v_{\sigma(i+1)} \otimes \cdots \otimes v_{\sigma(s)} )
\end{equation}
where $\operatorname{Sh}_{i,s-i}$ is the set of $(i,s-i)$ shuffles. The two-sided ideal $I \subseteq TV$ with $\Sym(V) = TV/I$ is also a coideal, so $\Sym(V)$ is a coalgebra with coproduct
\begin{equation}\label{eq:coalgebra_delta}
\Delta( v_1 \cdots v_s ) = \sum_{I \subseteq \{1,\ldots,s\}} \Big(\prod_{i \in I} v_i\Big) \otimes \Big(\prod_{j \notin I} v_j\Big)\,.
\end{equation}

\begin{lemma}\label{lemma:relate_to_fock} There is a canonical isomorphism of coalgebras
\begin{equation}
{!} V \lto \bigoplus_{P \in V} \Sym_P(V)
\end{equation}
where $\Sym_P(V) = \Sym(V)$. If we define
\[
c_P: \Sym_P(V) \lto V, \qquad c_P( v_0, v_1, v_2, \ldots ) = v_0 P + v_1
\]
then the isomorphism identifies the dereliction map on ${!} V$ with $\sum_P c_P$.
\end{lemma}
\begin{proof}
It is clear from \eqref{eq:coalgebra_delta} and \eqref{eq:lc_coalgebra_1} that Lemma \ref{lemma:iso_fock} is an isomorphism of coalgebras. The rest is immediate from Lemma \ref{lemma:dereliction_describe}.
\end{proof}

\begin{remark}\label{remark_additive_iso} Suppose that $V = V_1 \oplus V_2$. There is an isomorphism of coalgebras
\begin{gather*}
\Psi: {!} V_1 \otimes {!} V_2 \lto {!} V\,,\\
\Psi\left( \ket{\nu_1,\ldots,\nu_s}_P \otimes \ket{\omega_1,\ldots,\omega_t}_Q \right) = \ket{\nu_1,\ldots,\nu_s,\omega_1,\ldots,\omega_t}_{(P,Q)}
\end{gather*}
as may be verified from Lemma \ref{lemma:relate_to_fock} using $\Sym(V_1) \otimes \Sym(V_2) \cong \Sym(V)$. In this way various properties of ${!} V$ may be reduced to the one-dimensional case
\[
{!} ( k \cdot e ) = \bigoplus_{P \in k} k[x_P] = \bigoplus_{P \in k} k \oplus k \cdot x_P \oplus k \cdot x_P^2 \oplus \cdots\,.
\]
\end{remark}

The proof of the remaining results in this section are left to Appendix \ref{appendix:proofs}. Recall that all coalgebras are coassociative, cocommutative and counital and that $V$ is finite-dimensional.

\begin{theorem}\label{theorem:main} The dereliction $d: {!}V \lto V$ is universal among linear maps from coalgebras to $V$. That is, if $C$ is a coalgebra and $\phi: C \lto V$ there is a unique morphism of coalgebras $\Phi: C \lto {!}V$ making the diagram
\begin{equation}\label{eq:defining_philift2}
\xymatrix@C+2pc@R+2pc{
C \ar[r]^-{\phi}\ar[dr]_-{\Phi} & V\\
& {!}V \ar[u]_-{d}
}
\end{equation}
commute.
\end{theorem}

To describe the lifting $\Phi$ more explicitly, notice that for a coalgebra $C$ with coproduct $\Delta$ and counit $\varepsilon$ we may construct a map for any $l \ge 0$ from $C$ to $V^{\otimes l+1}$ as follows:
\begin{equation}\label{eq:stick_of_a_map}
\xymatrix@C+2pc{
C \ar[r]^-{\Delta^l} & C^{\otimes (l+1)} \ar[r]^-{\phi^{\otimes (l+1)}} & V^{\otimes (l+1)}\,.
}
\end{equation}
For $l = 0$ this is just $\phi$. Our convention is that $\Delta^{-1}$ denotes the counit so that \eqref{eq:stick_of_a_map} also makes sense for $l = -1$. For any $\eta \in C$ we get an element
\begin{equation}\label{eq:sequence_of_powers}
\sum_{l \ge -1} \phi^{\otimes (l+1)} \circ \Delta^l(\eta) = ( \varepsilon(\eta), \phi(\eta), \phi^{\otimes 2} \Delta(\eta), \ldots ) \in \prod_{l \ge 0} (V^{\otimes l})^{\mathfrak{S}_l}
\end{equation}
where $(V^{\otimes l})^{\mathfrak{S}_l}$ denotes tensors invariant under the action of the symmetric group. Any element of $\Sym(V^*)$ may be contracted against such a sequence to yield a scalar. With the notation of the theorem:

\begin{proposition}\label{prop:respairing} For $\eta \in C$, $\Phi(\eta)$ is the unique element of ${!}V$ with the property that
\begin{equation}\label{eq:respairing}
\Res\,( f \Phi(\eta) ) = f\; \lrcorner \;\sum_{l \ge -1} \phi^{\otimes (l+1)} \Delta^l(\eta)
\end{equation}
for every $f \in \Sym(V^*)$.
\end{proposition}

Note that in \eqref{eq:respairing} we use the $\Sym(V^*)$-module structure on ${!} V$ which follows from its definition as a local cohomology module over $R = \Sym(V^*)$. Given a set $X$ we denote the set of all partitions of the set by $\cat{P}_X$.

\begin{theorem}\label{theorem:describe_lifting} Let $W, V$ be finite-dimensional vector spaces and $\phi: {!}W \lto V$ a linear map. The unique lifting of $\phi$ to a morphism of coalgebras $\Phi: {!}W \lto {!}V$ is defined for vectors $P, \nu_1,\ldots,\nu_s \in W$ by
\begin{equation}\label{eq:describe_lift}
\Phi \ket{ \nu_1, \ldots, \nu_s }_P = \sum_{C \in \cat{P}_{\{1,\ldots,s\}}} \Big|\, \phi \ket{\nu_{C_{1}} }_P\,, \ldots \,, \phi \ket{ \nu_{C_{l}}}_P \Big\rangle_Q
\end{equation}
where $Q = \phi \vacu_P$ and $C = (C_1,\ldots,C_l)$.
\end{theorem}

In the formula \eqref{eq:describe_lift} the right hand side uses the notation $\nu_I$ of \eqref{eq:nu_I} for a set of indices $I$. Observe that as a special case of this formula, $\Phi \vacu_P = \vacu_Q$.

\begin{corollary}\label{corollary:bang_func} Let $\psi: W \lto V$ be a morphism of finite-dimensional vector spaces. The associated morphism of coalgebras $!\psi: {!}W \lto {!}V$ is given by
\begin{equation}\label{eq:bang_func}
!\psi \ket{\nu_1, \ldots, \nu_s}_P = \Big|\, \psi(\nu_{1}), \ldots, \psi(\nu_{s}) \Big\rangle_{\psi(P)}
\end{equation}
\end{corollary}

Even in the case where $V, W$ are infinite-dimensional the same formulas \eqref{eq:describe_lift} and \eqref{eq:bang_func} still apply. This follows from the discussion in Section \ref{section:infinite_dim} below.

\subsection{The infinite-dimensional case}\label{section:infinite_dim}

Let $V$ be a $k$-vector space, not necessarily finite-dimensional. Let $\{ V_i \}_{i \in I}$ be the set of finite-dimensional subspaces, so that $V = \varinjlim_i V_i$. For an inclusion $V_i \subseteq V_j$ the induced morphism of coalgebras ${!}V_i \lto {!}V_j$ is described by Corollary \ref{corollary:bang_func} as
\[
!V_i \lto {!}V_j, \qquad \ket{\nu_1,\ldots,\nu_s}_P \longmapsto \ket{\nu_1,\ldots,\nu_s}_P\,.
\]
Direct limits in the category of coalgebras may be taken in the category of vector spaces, that is, the direct limit in the category of vector spaces may be canonically equipped with a coalgebra structure in such a way that it is the colimit in the category of coalgebras. We may therefore define a coalgebra
\[
!V = \varinjlim_i {!}V_i\,.
\]
As a vector space this is the set of equivalence classes of pairs $(W, \eta)$ consisting of a finite-dimensional $W \subseteq V$ and $\eta \in {!}W$. Thus elements of ${!} V$ may still be described as linear combinations of kets $\ket{\nu_1,\ldots,\nu_s}_P$ for $P, \nu_1,\ldots,\nu_s \in V$. The dereliction map ${!} V \lto V$ and counit ${!} V \lto k$ are induced out of the colimit by the corresponding maps on the ${!}V_i$.

\begin{proposition}\label{prop:infinite_v_lifting} The dereliction map $d: {!}V \lto V$ is universal among linear maps to $V$ from coalgebras.
\end{proposition}
\begin{proof}
If $C$ is a coalgebra and $\phi: C \lto V$ a linear map, then for each finite-dimensional subcoalgebra $C' \subseteq C$ we may factor $\phi|_{C'}$ through a finite-dimensional subspace $V_i \subseteq V$ and the map $C' \lto V_i$ lifts to a coalgebra morphism $C' \lto {!} V_i \lto {!} V$. These assemble to a morphism of coalgebras $C \lto {!} V$ with the right property.
\end{proof}

\appendix

\section{Proofs}\label{appendix:proofs}

This appendix contains the proofs of some results in Section \ref{section:expmod}. The notation is as given there, so in particular we have a fixed finite-dimensional $k$-vector space $V$ of dimension $n$ and $R = \Sym(V^*)$. For each maximal ideal $\mf{m}$ there is a restriction map along $R \lto R_{\mf{m}}$ 
\begin{equation}\label{eq:restrict_map_cont}
\Hom_k^{\cont}(R_{\mf{m}},k) \lto \Hom_k^{\cont}(R,k)\,.
\end{equation}
The image is the subspace of all linear $\mu: R \lto k$ vanishing on a power of $\mf{m}$.

\begin{lemma}\label{lemma:decomp_contmap} The canonical map
\begin{equation}\label{eq:decomp_contmap}
\bigoplus_{\mf{m}} \Hom_k^{\cont}(R_{\mf{m}},k) \lto \Hom_k^{\cont}(R,k)
\end{equation}
is an isomorphism of $R$-modules, where $\mf{m}$ ranges over the maximal ideals of $R$. In particular, there is an isomorphism of coalgebras
\begin{gather}
{!} V \lto \Hom_k^{\cont}(R,k), \label{eq:decomp_contmap2}\\
\eta \longmapsto \Res\;( (-) \cdot \eta )\,.\nonumber
\end{gather}
\end{lemma}
\begin{proof}
The isomorphism \eqref{eq:decomp_contmap2} is immediate from Theorem \ref{theorem:lcwithcontmap} once we establish the first claim. Since $R/\mf{m}^i \cong R_{\mf{m}}/\mf{m}^iR_{\mf{m}}$ the map \eqref{eq:restrict_map_cont} is injective. Given a continuous function $\mu: R \lto k$ we may define $I$ to be the largest ideal contained in $\Ker(\mu)$. Suppose that $\mu$ vanishes on some power of a maximal ideal $\mf{m}$, say $\mf{m}^i \subseteq I$. If $\mf{n}$ is another maximal ideal and $I \subseteq \mf{n}$ then $\mf{m}^i \subseteq \mf{n}$ implies $\mf{m} = \mf{n}$. Hence either $I \subseteq \mf{m}$ or $I = R$. This means that $\mu$ can only vanish on a power of a different maximal ideal if $\mu$ is identically zero, which proves that the map \eqref{eq:decomp_contmap} is injective.

So it only remains to prove that the map \eqref{eq:decomp_contmap} is surjective. This follows from the Chinese remainder theorem. Given $\mu$ let the ideal $I \subseteq R$ be defined as above. Then $V(I)$, the set of maximal ideals containing $I$, is $\{ \mf{m}_1,\ldots,\mf{m}_r \}$ for some $r$. The ideal $I$ has a minimal primary decomposition $I = \mf{q}_1 \cap \cdots \cap \mf{q}_r$ with $\mf{q}_i$ an $\mf{m}_i$-primary ideal, so $\mf{m}_i^{b_i} \subseteq \mf{q}_i$ for some $b_i$. It follows that $\mf{m}_1^{b_1} \cdots \mf{m}_r^{b_r} \subseteq I$. That is, $\mu$ factors through $R/\mf{m}_1^{b_1} \cdots \mf{m}_r^{b_r}$. By the Chinese remainder theorem there is an isomorphism of $k$-algebras
\[
R/\mf{m}_1^{b_1} \cdots \mf{m}_r^{b_r} \cong R/\mf{m}_1^{b_1} \oplus \cdots \oplus R/\mf{m}_r^{b_r}
\]
and so $\mu$ may be written as a sum of linear maps $R/\mf{m}_i^{b_i} \lto k$. Hence $\mu$ belongs to the image of \eqref{eq:decomp_contmap}, and the proof is complete.
\end{proof}

\begin{proof}[Proof of Theorem \ref{theorem:main}] This is well-known for $\Hom_k^{\cont}(R,k)$ \cite[Section 6]{sweedler} and in light of the isomorphism \eqref{eq:decomp_contmap2} it also holds for ${!} V$. We recall the proof: since $C$ is a direct limit of its finite-dimensional subcoalgebras (see \cite[\S 2.2]{sweedler}) we may reduce to the case of $C$ finite-dimensional so that $C^*$ is canonically a commutative $k$-algebra. Then the dual $\phi^*: V^* \lto C^*$ induces a morphism of $k$-algebras $\psi_1: \Sym(V^*) \lto C^*$. This map must vanish on some ideal of finite codimension $I$, and the factorisation we denote by $\psi_2: \Sym(V^*)/I \lto C^*$. The dual gives a morphism of coalgebras
\begin{equation}\label{eq:main3}
\xymatrix{
C \cong C^{**} \ar[r]^-{\psi_2^*} & \big( \Sym(V^*)/I \big)^*
}
\end{equation}
and since $\Hom_k( R/I, k )$ is a subcoalgebra of $\Hom^{\cont}_k(R,k)$ this defines the necessary lifting $\Phi: C \lto \Hom^{\cont}_k(R,k)$.
\end{proof}

\begin{proof}[Proof of Proposition \ref{prop:respairing}] This follows from the proof of Theorem \ref{theorem:main}. Again we may assume $C$ finite-dimensional, and in light of \eqref{eq:decomp_contmap2} it suffices to prove that for $\eta \in C$ the lifting $\Phi: C \lto \Hom^{\cont}_k(R,k)$ is defined by $\Phi( \eta )(f) = f \contract \sum_{l \ge -1} \phi^{\otimes(l+1)} \Delta^l( \eta )$.

We may assume that $f$ is a monomial $\kappa_1 \otimes \cdots \otimes \kappa_q$. Then, using Sweedler notation,
\begin{equation}\label{eq:blahdyblah}
f \contract \sum_{l \ge -1} \phi^{\otimes (l+1)} \Delta^l( \eta ) = \sum \kappa_1 \phi( \eta_{(1)} ) \cdots \kappa_q \phi( \eta_{(q)} )\,.
\end{equation}
On the other hand, using the notation of \eqref{eq:main3}
\begin{align*}
\Phi( \eta )(f) &= \psi_2^*( \eta )( f ) = \psi_2( \kappa_1 \otimes \cdots \otimes \kappa_q )( \eta )
= \left\{ \psi_1( \kappa_1 ) \bullet \cdots \bullet \psi_1( \kappa_q ) \right\}(\eta)
\end{align*}
where $\bullet$ denotes the product in $C^*$. But this agrees with the right hand side of \eqref{eq:blahdyblah}, since $\psi_1(\kappa_i) = \kappa_i \circ \phi$.
\end{proof}

\begin{proof}[Proof of Theorem \ref{theorem:describe_lifting}] For $V$ we take the coordinates\footnote{Using Remark \ref{remark_additive_iso} we can reduce to the case of $V$ one-dimensional, but this does not significantly simplify the proof.} given in Example \ref{example:coordinates_V}. We also drop the subscripts from kets. Using contraction against the tuple of \eqref{eq:sequence_of_powers} we define
\begin{gather*}
\mathscr{C}: \Sym(V^*) \times \LC( W, P ) \lto k\\
\mathscr{C}(f, \eta) = f \; \lrcorner \; \sum_{l \ge -1} \phi^{\otimes (l+1)} \Delta^l( \eta )\,.
\end{gather*}
We define a second function with the same domain and codomain by
\[
\mathscr{D}(f, \ket{\nu_1, \ldots, \nu_s}) = \sum_{C \in \cat{P}_{\{1,\ldots,s\}}} \phi \ket{\nu_{C_1}} \cdots \phi \ket{\nu_{C_l}}(f)\arrowvert_{x = Q}\,.
\]
Here $\phi \ket{\nu_{C_1}} \cdots \phi \ket{\nu_{C_l}}$ is an element of $V^{\otimes l}$ which is mapped to a differential operator on $R$. After acting with this operator on $f$, the resulting polynomial is evaluated at the point $Q$. The proof is divided into two steps.

\emph{Step 1} (Prove that $\mathscr{C} = \mathscr{D}$). There are some simple cases where this is obvious, for instance $\cat{C}(f, \vacu) = f( Q )$ and $\cat{C}( 1, \ket{\nu_1,\ldots,\nu_s} ) = \Res\;\ket{\nu_1,\ldots,\nu_s} = \delta_{s = 0}$ and the same is true of $\mathscr{D}$. We reduce to these cases by recursion formulas, as follows.

For any index $1 \le a \le n$ we will prove that for $\mathscr{E} \in \{ \mathscr{C}, \mathscr{D} \}$ and $g \in \Sym(V^*)$
\begin{equation}\label{eq:recur_formula}
\mathscr{E}( x_a g, \ket{\nu_1,\ldots,\nu_s} ) = \sum_{I \subseteq \{ 1,\ldots,s\}} \phi\ket{\nu_I}^a \mathscr{E}( g, \ket{\nu_{I^c}})
\end{equation}
where for $v \in V$ we write $v = \sum_a v^a e_a$ using the chosen basis $e_1,\ldots,e_n$ for $V$.

It then follows by induction on the degree of $f$ that $\mathscr{C} = \mathscr{D}$. First we prove the recursion identity \eqref{eq:recur_formula} for $\mathscr{C}$. In what follows $I$ always ranges over subsets of $\{ 1, \ldots, s \}$ including the empty set. Define $\eta = \ket{\nu_1,\ldots,\nu_s}$. Since
\[
\phi^{\otimes(l+1)} \Delta^l( \eta ) = \sum_I \phi \ket{\nu_I} \otimes \phi^{\otimes l} \Delta^{l-1} \ket{\nu_{I^c}}
\]
we have, using that $f = x_a g$ has degree $\ge 1$,
\begin{align*}
\mathscr{C}( x_a g, \eta ) &= f \contract \sum_{l \ge 0} \phi^{\otimes (l+1)} \Delta^l ( \eta )\\
&= f \contract \Big( \sum_{l \ge 1} \phi^{\otimes (l+1)} \Delta^l( \eta ) + \phi( \eta ) \Big)\\
&= f \contract \phi( \eta ) + f \contract \sum_{l \ge 1} \sum_I \phi \ket{\nu_I} \otimes \phi^{\otimes l} \Delta^{l-1} \ket{\nu_{I^c}}
\end{align*}
Now if $g_0$ denotes the constant term of $g$, $f \contract \phi( \eta ) = g_0 \cdot \phi( \eta )^a$, and
\begin{align*}
f \contract \sum_{l \ge 1} \sum_I \phi \ket{\nu_I} \otimes \phi^{\otimes l} \Delta^{l-1} \ket{\nu_{I^c}} &= \sum_I \sum_{l \ge 0} ( x_a \contract \phi\ket{\nu_I} ) \cdot ( g \contract \phi^{\otimes (l+1)} \Delta^l \ket{\nu_{I^c}})\\
&= \sum_I \sum_{l \ge 0} \phi\ket{\nu_I}^a \cdot ( g \contract \phi^{\otimes (l+1)} \Delta^l \ket{\nu_{I^c}})
\end{align*}
so that
\begin{align*}
\mathscr{C}( x_a g, \eta ) &= g_0 \phi( \eta )^a + \sum_{I \neq \{1,\ldots,s\}} \phi \ket{\nu_I}^a \cdot ( g \contract \sum_{l \ge -1} \phi^{\otimes (l+1)} \Delta^l \ket{\nu_{I^c}})\\
&+ \phi( \eta )^a \cdot ( g \contract \sum_{l \ge 0} \phi^{\otimes (l+1)} \Delta^l \ket{\nu_{I^c}})\\
&= \phi( \eta )^a \mathscr{C}(g, \vacu) + \sum_{I \neq \{1,\ldots,s\}} \phi \ket{\nu_I} \mathscr{C}(g, \ket{\nu_{I^c}})\\
&= \sum_{I \subseteq \{ 1,\ldots,s\}} \phi\ket{\nu_I}^a \mathscr{C}( g, \ket{\nu_{I^c}})\,.
\end{align*}
This proves the recursion identity for $\mathscr{C}$. 

For $\mathscr{D}$ we may compute as follows. A sum over the multiindex $\bold{i}$ stands for a sum over indices $i_1,\ldots,i_l$, where $l$ stands for the length of a particular partition $C \in \cat{P} = \cat{P}_{\{1,\ldots,s\}}$. For such a partition we define the differential operator
\[
D( C, \eta ) = \phi\ket{\nu_{C_1}} \cdots \phi\ket{ \nu_{C_l} }\,.
\]
Then we may compute that
\begin{align*}
D( C, \eta )( x_a g ) &= \sum_{\bold{i}} \phi\ket{\nu_{C_1}}_{i_1} \cdots \phi\ket{\nu_{C_l}}_{i_l} \frac{\partial}{\partial x_{i_1}} \cdots \frac{\partial}{\partial x_{i_l}}( x_a g )\\
&= \sum_{\bold{i}} \phi\ket{\nu_{C_1}}_{i_1} \cdots \phi\ket{\nu_{C_l}}_{i_l} \frac{\partial}{\partial x_{i_1}} \cdots \frac{\partial}{\partial x_{i_{l-1}}}( \delta_{a = i_l} g + x_a \frac{\partial}{\partial x_{i_l}} g )\\
&= \cdots\\
&= \sum_{\bold{i}} \phi\ket{\nu_{C_1}}_{i_1} \cdots \phi\ket{\nu_{C_l}}_{i_l} \Big\{ \sum_{j=1}^l \delta_{i_j = a}  \frac{\partial}{\partial x_{i_1}} \cdots \widehat{ \frac{\partial}{\partial x_{i_j}}} \cdots \frac{\partial}{\partial x_{i_l}}(g)\\
&+ x_a \frac{\partial}{\partial x_{i_1}} \cdots \frac{\partial}{\partial x_{i_l}}(g) \Big\}\\
&= x_a D( C, \eta )(g) + \sum_{j=1}^l \phi \ket{\nu_{C_j}}^a D( C \setminus C_j, \ket{\nu_{C_j^c}} )(g)
\end{align*}
When we sum over all partitions and evaluate at $Q$,
\begin{align*}
\mathscr{D}( x_a g, \eta ) &= Q^a \mathscr{D}( g, \eta ) + \sum_C \sum_{j=1}^l \phi \ket{\nu_{C_j}}^a D( C \setminus C_j, \ket{\nu_{C_j^c}})(g)|_{x = Q}\\
&= Q^a \mathscr{D}( g, \eta ) + \sum_{\emptyset \neq I \subseteq \{1,\ldots,s\}} \sum_{C \in \cat{P}_{I^c}} \phi\ket{\nu_I}^a D( C, \ket{\nu_{I^c}} )(g)|_{x = Q}\\
&= Q^a \mathscr{D}( g, \eta ) + \sum_{\emptyset \neq I \subseteq \{1,\ldots,s\}} \phi \ket{\nu_I}^a \mathscr{D}(g, \ket{\nu_{I^c}})\\
&= \sum_I \phi\ket{\nu_I}^a \mathscr{D}(g, \ket{\nu_{I^c}})\,.
\end{align*}
This proves the recursion identity for $\mathscr{D}$ also, whence $\mathscr{C} = \mathscr{D}$.

\emph{Step 2.} By Proposition \ref{prop:respairing}, $\Phi( \eta ) \in {!}V$ is unique such that $\Res\;( f \Phi( \eta ) ) = \mathscr{C}(f, \eta)$ for every $f \in \Sym(V^*)$. But given our identification of $\mathscr{C}$ with $\mathscr{D}$ and Lemma \ref{lemma:residue_differentiates}, it is clear that this unique element $\Phi( \eta )$ must be as described in the statement of the theorem.
\end{proof}

\section{Sweedler's approach}\label{section:compare_sweedler}

In this appendix we show how to derive the presentation of the universal cocommutative coalgebra ${!} V$ in Lemma \ref{lemma:relate_to_fock} from the structure theory of coalgebras in \cite{sweedler}. Throughout $k$ is an algebraically closed field of any characteristic, and coalgebras are counital and coassociative but not necessarily cocommutative.

Let $(C, \Delta, \varepsilon)$ be a nonzero coalgebra. Recall from \cite[Chapter 8]{sweedler} that $C$ is \emph{irreducible} if any two non-zero subcoalgebras have a non-zero intersection, \emph{simple} if it has no non-zero proper subcoalgebras and \emph{pointed} if all simple subcoalgebras of $C$ are $1$-dimensional. The symmetric algebra $\Sym(V)$ is a pointed irreducible cocommutative coalgebra. By \cite[Lemma 8.0.1]{sweedler} since $k$ is algebraically closed, any cocommutative coalgebra is pointed. A subcoalgebra $D$ of $C$ is an \emph{irreducible component} (IC) if it is a maximal irreducible subcoalgebra. By \cite[Theorem 8.0.5]{sweedler} any cocommutative coalgebra $C$ is the direct sum of its irreducible components.

If $0 \neq c \in C$ with $\Delta(c) = c \otimes c$ then $\varepsilon(c) = 1$ and
\[
G(C) = \{ 0 \neq c \in C \l \Delta(c) = c \otimes c \}\,.
\]

In \cite[Theorem 6.4.3]{sweedler} the commutative cofree coalgebra $C(V)$ on a vector space $V$ is constructed as a subcoalgebra of the cofree coalgebra. Moreover it is shown that if $V_1, V_2$ are vector spaces then there is a canonical isomorphism of coalgebras \cite[Theorem 6.4.4]{sweedler}
\[
C(V_1 \oplus V_2) \cong C(V_1) \otimes C(V_2)\,.
\]
This may be used to make $C(V)$ into a Hopf algebra: the maps
\begin{align*}
V \oplus V \lto V, &\qquad (v,w) \longmapsto v + w\,,\\
V \lto V, &\qquad v \longmapsto -v\,,\\
0 \lto V, &\qquad 0 \longmapsto 0\,.
\end{align*}
induce coalgebra morphisms
\begin{gather*}
M: C(V) \otimes C(V) \cong C(V \oplus V) \lto C(V)\,,\\
S: C(V) \lto C(V),\\
u: k \cong C(0) \lto C(V)
\end{gather*}
which make $C(V)$ into a commutative Hopf algebra with antipode $S$ \cite[Theorem 6.4.8]{sweedler}. The universal property of $C(V)$ yields an isomorphism of vector spaces
\begin{equation}\label{eq:blah_sweedler}
G(C(V)) \cong \Hom_{\operatorname{Coalg}}(k, C(V)) \cong \Hom_k(k, V) \cong V\,.
\end{equation}
For $P \in V$ we write $\vacu_P$ for the corresponding element of $G(C(V))$. In particular the image of the unit map $u$ is $k \cdot \vacu_0 \in C(V)$. We denote by $C(V)^P$ the irreducible component of $C(V)$ containing $\vacu_P$. Then by \cite[Proposition 8.1.2]{sweedler}, using the identification of \eqref{eq:blah_sweedler}, there is a canonical isomorphism of coalgebras
\[
C(V) \cong \bigoplus_{P \in V} C(V)^P\,.
\]
It remains to show that $C(V)^P \cong \Sym(V)$ as coalgebras (at least in the case of characteristic zero) and to derive the explicit form of the universal map $C(V) \lto V$.

In \cite[Chapter 12]{sweedler} there is a construction of the universal pointed irreducible coalgebra $\sh(V)$ (the \emph{shuffle algebra}) mapping to a finite-dimensional space $V$ \cite[Lemma 12.0.1]{sweedler}. In \cite[Section 12.2]{sweedler} the cofree pointed irreducible cocommutative coalgebra $B(V)$ over $V$ is defined as the maximal cocommutative subcoalgebra of $\sh(V)$. This is a graded sub-bialgebra of $\sh(V)$ and hence $B(V)_0 = k$ and $B(V)_1 = V$. Next we show to relate $B(V)$ to the irreducible components $C(V)^P$ of $C(V)$. 

The projection onto degree one is a map $B(V) \lto V$ which must factor as
\[
\xymatrix{
B(V) \ar[r]^-{\iota} & C(V) \ar[r] & V
}
\]
for some morphism of coalgebras $\iota$. By \cite[Theorem 12.2.6]{sweedler} this map $\iota$ is injective, and since $\vacu_0$ lies in the image it must induce an isomorphism of coalgebras $B(V) \cong C(V)^0$. At this point we must use the structure theory of Hopf modules from \cite[Chapter 4]{sweedler}, as described on \cite[p.175]{sweedler}, to see that there is an isomorphism of coalgebras
\[
C(V)^0 \otimes_k kG \lto C(V), \qquad h \otimes g \longmapsto h \cdot g
\]
where $kG$ is the linear span of $G = G(C(V))$ within $C$ with its (trivial) coalgebra structure. Here $h \cdot g$ denotes the product in the Hopf algebra. 

This gives us a canonical isomorphism of coalgebras $C(V)^0 \cong C(V)^P$ for any $P \in V$. With $C = C(V)^P$ the map $C \lto C(V) \lto V$ induces $C^+ \lto V$ (where $C^+$ is the kernel of the counit) which by the universal property of $B(V)$ \cite[Theorem 12.2.5]{sweedler} can be proven to be induced on $C^+$ by the inverse of the isomorphism $B(V) \cong C$ composed with $B(V) \lto V$. Thus it only remains to describe $C \lto V$ on $1 \in k \cong C_0$. But it follows from \eqref{eq:blah_sweedler} that the value of this map on $1$ must be $P$ itself.

So far everything we have discussed is characteristic free. If we specialise to characteristic zero then $B(V)$ is isomorphic as a coalgebra to $\Sym(V)$ by the results of \cite[Chapter 12]{sweedler} and we recover the description of $C(V)$ in Lemma \ref{lemma:relate_to_fock}. That is, ${!} V \cong C(V)$ matches $\LC(V,P)$ with the component $C(V)^P$. 

\bibliographystyle{amsalpha}
\providecommand{\bysame}{\leavevmode\hbox to3em{\hrulefill}\thinspace}
\providecommand{\href}[2]{#2}

\end{document}